\documentclass[12pt,a4paper]{amsart}
\usepackage{amsmath,amssymb,amsthm,mathrsfs,graphicx,mathtools,units}

\allowdisplaybreaks[1]

\usepackage{amsfonts}
\usepackage{marginnote}
\usepackage{color}
\usepackage{hyperref}

\newtheorem{thm}{Theorem}

\newtheorem{lem}[thm]{Lemma}
\newtheorem{cor}[thm]{Corollary}
\theoremstyle{remark}
\theoremstyle{definition}

\theoremstyle{remark}
\newtheorem{rem}[thm]{Remark}

\numberwithin{equation}{section}
\numberwithin{thm}{section}

\newcommand{\eps}{\varepsilon}

\newcommand{\dt}{\delta}

\newcommand{\R}{\mathbb{R}}

\newcommand{\cI}{\mathcal{I}}

\newcommand{\cL}{\mathcal{L}}

\newcommand{\DD}{\mathcal{D}}

\newcommand{\be}{\begin{equation}} 
\newcommand{\ee}{\end{equation}}
\newcommand{\bea}{\begin{eqnarray}} 
\newcommand{\eea}{\end{eqnarray}}
\newcommand{\bean}{\begin{eqnarray*}} 
\newcommand{\eean}{\end{eqnarray*}}
\newcommand{\rf}[1]{(\ref {#1})} 

\def\dx{\,{\rm d}x}
\def\dy{\,{\rm d}y}
\def\dt{\,{\rm d}t}
\def\dr{\,{\rm d}r}

\title[A diffusive aggregation model]{Concentration phenomena \\ in a diffusive aggregation model}

 \author[P. Biler]{Piotr Biler}
\address[P. Biler]{Instytut Matematyczny, Uniwersytet Wroc{\l}awski, pl. Grunwaldzki 2/4, 50-384  Wroc{\l}aw, Poland}
\email{piotr.biler@uwr.edu.pl}

 \author[A. Boritchev]{Alexandre Boritchev}
 \address[A. Boritchev]{Universit\'e Claude Bernard -- Lyon 1, CNRS UMR 5208, Institut Camille Jordan, F-69622 Villeurbanne, France}
\email{alexandre.boritchev@gmail.com}

 \author[G. Karch]{Grzegorz Karch}
\address[G. Karch]{Instytut Matematyczny, Uniwersytet Wroc{\l}awski, pl. Grunwaldzki 2/4, 50-384 Wroc{\l}aw, Poland}
\email{grzegorz.karch@uwr.edu.pl}

 \author[Ph. Lauren{\c c}ot]{Philippe Lauren{\c c}ot}
 \address[Ph. Lauren{\c c}ot]{Institut de Math\'ematiques de Toulouse, UMR 5219, Universit\'e de Toulouse, CNRS, F-31062, Toulouse Cedex 9,
 France 
}
\email{Philippe.Laurencot@math.univ-toulouse.fr}

\begin{document}

\begin{abstract} 
We consider the  drift-diffusion equation
$u_t-\eps\Delta u+\nabla\cdot(u\nabla K\ast u)=0$
 in the whole space with global-in-time bounded solutions. 
Mass concentration phenomena for radially symmetric solutions of  this equation  with small diffusivity  are studied.  
\end{abstract}



\date{\today}
\maketitle

\baselineskip=15pt

\section{Introduction and motivations} \label{sec:int}
The nonlocal nonlinear evolution model 
$$
u_t-\eps\Delta u+\nabla\cdot(u\nabla K\ast u)=0, \qquad x\in \mathbb{R}^N, \;t>0, 
$$
 describes pairwise interaction of particles with the density $u=u(x,t)$ through the convolution with a kernel $K:\R^N\to \R$. These particles are also subject to the linear  diffusion represented by the Laplacian, with  the diffusion coefficient  $\eps> 0$.  
Such models appear in various physical and biological settings. 
For instance, Astrophysics is a source of mean-field models of gravitationally attracting particles going back to the famous Chandrasekhar equation for the equilibrium of radiating stars, see \cite{Cha,CSR}. 
Another source of related models is Mathematical Biology where chemotaxis (haptotaxis, angiogenesis, etc.)~phenomena for populations of either cells or (micro)organisms are described by various modifications of the Keller--Segel systems, see e.g.~\cite{HP09}. 

We supplement this equation with a nonnegative, radial, bounded and integrable  initial condition, and we assume that $\nabla K\in L^\infty(\mathbb{R}^N)$ which is a sufficient assumption  to guarantee that the corresponding Cauchy problem has a global-in-time regular solution, uniformly bounded for $t>0$. 
The main result reported in this work states that under a suitable assumption on the singularity of the kernel $K$ and for a small diffusion coefficient $\varepsilon>0$, one observes concentration phenomena of such global-in-time  solutions, namely, $\varepsilon$-small neighbourhoods of the origin carry an $\varepsilon$-uniform portion of the total mass. 
Results  reported in this work show what kind of concentration of solutions  can be expected in ``correct'' mathematical models (i.e.~those without finite time blowups) and how to detect them in experiments with  observations accumulated over a sufficiently long time interval.  

Our proofs are based on a new general methodology stemming from considerations  in \cite{KS10,B16ab,BKL09}.  
Note also that our study allows us to obtain estimates of small scale quantities such as structure functions which are relevant for hydrodynamical turbulence. 
A similar scheme has been studied in the context of the large scale evolution of  the Universe modeled by a multidimensional analogue of the classical Burgers equation 
 $u_t-\eps \Delta u+(u\cdot\nabla)u=0$.
In that context, a link to the phenomenological/formal Kolmogorov K41 theory of turbulence has been made in \cite{Bir01} and in the papers of the second author \cite{B14,B16ab,B18}  culminating in  the study of the asymptotic behaviour of solutions of arbitrary size in the vanishing diffusion limit in terms of various functional norms.

\section{Main results} \label{sec:mr}
Let us now state the results of our work in detail. We consider the Cauchy problem  
\begin{align}
u_t-\eps\Delta u+\nabla\cdot(u\nabla K\ast u)&=0, &&x\in\mathbb{R}^N,\; t>0,\label{agg-CP} \\
u(x,0)&=u_0(x),&&x\in\mathbb{R}^N,\label{i-CP}
\end{align}
with an initial condition  satisfying 
 \begin{equation}
 u_0\in L^1(\mathbb{R}^N)\cap L^{\infty}(\mathbb{R}^N), \quad u_0\ge 0, \quad M \equiv \int_{\mathbb{R}^N} u_0(x)\dx>0, \label{PhL10}
 \end{equation} 
and with a  constant diffusion coefficient $\eps>0$.  
 The interaction kernel $K=K(x)$ in equation \rf{agg-CP} is a radially symmetric  function such that
\be 
 K(x)=k(|x|)\ \ {\rm with}\ \ k\in{\mathcal C}^1(0,\infty),\ \ k' \in L^\infty(0,\infty),\label{KN1} 
\ee
 and  
\be
\begin{split}
 \kappa_\Lambda  \equiv  
- &\sup_{s\in (0,\Lambda)} k'(s)\in (0,\infty)\quad \text{for  each}\quad \Lambda>0, \quad\\ &\text{with} \quad  \kappa_0 \equiv  
\lim_{\Lambda\searrow 0}\kappa_\Lambda\in(0,\infty). \label{KN2}
\end{split}
\ee 
For $N=1$, we actually need something stronger than $k' \in L^\infty(0,\infty)$, namely 
\begin{equation} \label{KN3}
k'' \in L^1(0,\infty).
\end{equation} 
The functions $K(x)=-|x|$ and  $K(x)= {\mathrm e}^{-|x|}$  are our basic examples of  interaction kernels often used in applications, see  e.g.~\cite{BCL09}. 

Such kernels $K$ are {\em mildly singular}; that is, solutions of the Cauchy problem~\eqref{agg-CP}--\eqref{i-CP}  are global-in-time (see Remark \ref{rem1} below), even though interactions are strong enough to trigger finite time blowup in the  diffusion-free case $\eps=0$, see, e.g.,~\cite{BCL09, BLR11, BGL12} and the references therein.   

\begin{rem}\label{rem1}
It is rather standard to show that under  assumption \eqref{KN1}
(implying 
$\nabla K\in L^\infty(\R^N)$)
 and for each initial condition $u_0$ satisfying~\eqref{PhL10},
problem~\rf{agg-CP}--\rf{i-CP} 
  has a unique, nonnegative,   smooth, global-in-time solution which satisfies the sign conservation property: $u(x,t)\ge 0$, and the mass conservation property 
\be
\int_{\R^N}u(x,t)\dx=\int_{\R^N}u_0(x)\dx = M
\qquad \text{for all} \quad t\geq 0.
\label{m-cons}
\ee
This solution decays sufficiently fast for large values of $|x|$ such that  integration by parts in the following sections is fully justified.
By the uniqueness, this solution is radial in $x$ if the corresponding initial condition is.
The construction of such solutions is usually performed in the framework of mild ones, and their regularity is shown afterwards up to the classical smoothness.
More generally, according to \cite{KS11}, if $\nabla  K \in L^p(\R^N)$ with some $p\in (N,\infty]$, then problem~\eqref{agg-CP}--\eqref{i-CP} has a unique, regular, nonnegative, radial,  global-in-time solution for every nonnegative radial initial condition $u_0\in L^1(\R^N)\cap L^q(\R^N)$ for some $q>1$. 
On the other hand, it is well-known that for more singular kernels, for example when $K=-E_N$ is the fundamental solution of the Laplacian in $\R^N$  with $N\geq 2$
as in the parabolic-elliptic Keller-Segel model of chemotaxis, some solutions blow up in a finite time.
For a classification of kernels based on the property whether all solutions of equation \rf{agg-CP} are global-in-time or they can blow up in a finite time, see \cite{BW99, BRB11,  KS11,LS18}.
\end{rem}

\begin{rem}\label{rem2} 
Solutions to  the Cauchy problem~\eqref{agg-CP}--\eqref{i-CP} with $\varepsilon >0$ and under assumptions~\eqref{PhL10}--\eqref{KN1}  are not only global-in-time but also uniformly bounded in time (but not in $\varepsilon$). Indeed, we prove below in Lemma 
\ref{Lpupper} that $\sup_{t>0} \|u(t)\|_p<\infty$ for each $p\in [1,\infty)$.
This bound for the $L^\infty$-norm is obtained immediately from estimates of the Sobolev norms in Lemma \ref{H1upper} in the case  $N=1$. Such an estimate for the $L^\infty$-norm in  higher dimensional case can be shown in an analogous way and we shall discuss it in the forthcoming paper \cite{BBKL20}.
\end{rem}

Our goal is to study the behaviour of the family of solutions to problem~\eqref{agg-CP}--\eqref{i-CP} when $\eps >0$  is small. 

\begin{thm}[Concentration of mass at the origin]\label{conc}
Let $u=u(x,t)$ be a radial,  nonnegative, global-in-time solution to problem~\rf{agg-CP}--\rf{i-CP} with $\eps>0$, and with the interaction kernel $K$ satisfying assumptions~\rf{KN1}--\rf{KN3}. 
Assume that the radial initial condition $u_0$  satisfies conditions~\eqref{PhL10} and, moreover, suppose that  there exists   $\Lambda>0$ such that 
\begin{equation}
\mu_\Lambda\equiv  \int_{\mathbb{R}^N} \min\{|x|,\Lambda\} u_0(x)\dx < \frac{\kappa_\Lambda M}{4 \left( \kappa_\Lambda + 2 \|k'\|_\infty \right)}  \Lambda. \label{mu-Lambda}
\end{equation} 
Then for some explicit numbers  $\eps_\ast>0$,  $T_\ast>0$, $C_\ast>0$, and $\lambda>0$, independent of $\eps$, the following inequality holds true
\be
\int_0^{T_\ast}\int_{B_{\lambda\eps}}u(x,t)\dx \dt\ge C_\ast
\qquad 
\text{for all} \quad  \eps\in(0,\eps_\ast). 
\label{concen}
\ee
\end{thm} 

We also derive analogous estimates for $L^p$-norms of solutions on small balls.

\begin{cor} \label{c2}
Let $p\in [1,\infty)$. Under the assumptions of Theorem~\ref{conc} and using the same notation, the solution $u$ of problem~\rf{agg-CP}--\rf{i-CP} satisfies
\begin{equation} \label{conc2bis}
\int_0^{T_*} \left( \int_{B_{\lambda \varepsilon}} u(x,t)^p\dx \right)^{1/p} \dt\geq C_{**}(p) \varepsilon^{-\frac{N(p-1)}p}.
\end{equation} 
Here, the number $C_{**}(p)$ depends on the same parameters as $C_*$ in Theorem~\ref{conc} as well as on $p$.
\end{cor}

Theorem~\ref{conc} together with Corollary~\ref{c2} signify that, even if the  interactions described by the attractive kernel do 
not  lead to a formation of singularities for the solution $u(t)$ neither in finite nor  in the infinite time for $\eps>0$ (cf.~Remarks~\ref{rem1} and \ref{rem2}),
 there are concentration phenomena of solutions on $\varepsilon$-small balls for $0<\eps \ll 1$. 
To the best of our knowledge, this is the first result where the small diffusion asymptotics is analyzed in a sharp way for an aggregation model. Here, we have been motivated by results of this type 
obtained  in the specific case of 
the generalized Burgers equation  in \cite{B14,B16ab,B18}.

The proofs of Theorem~\ref{conc} and of Corollary~\ref{c2}
are based on an analysis of the time evolution  of the quantity ${\mathcal D}_\Lambda(u(t))$ defined below in \rf{D} which is related to the mass concentration of $u(t)$ at the origin. Then, in Section~\ref{sec:Lp}, we prove upper bounds for the norms in $L^p$ (and in $H^1$ for $N=1$) of solutions. Using these estimates in Section~\ref{sec:proof},
we derive lower bounds for the $L^p$-norms as 
stated in Theorem~\ref{conc} and  Corollary~\ref{c2}.

\begin{rem}\label{c3}
The order of  growth $\varepsilon^{-N(p-1)/p}$ of the $L^p$-norms stated in inequality~\eqref{conc2bis} is optimal because we have also an analogous (but without time average or spatial localisation) upper estimate. More precisely, 
we prove in Lemma~\ref{Lpupper} below that 
for each $p \in [1,\infty)$ 
there exists a number $C(p,u_0)>0$ such that
$$
\|u(t)\|_p \leq \ C(p,u_0)  \eps^{-N(p-1)/p} \qquad \text{for all } \quad t>0
$$
and for all sufficiently small $\varepsilon>0$.
This is a genuinely nonlinear effect  since such estimates  of $L^p$-norms 
for solutions of the heat equation $w_t=\eps\Delta w$ are different. 
Indeed, it follows from the explicit form of solutions via the convolution with  the Gauss--Weierstrass kernel that
$$
\|w(t)\|_p\asymp (\eps t)^{-N(p-1)/2p}\|w(0)\|_1.
$$ 
\end{rem}

\begin{rem}
Generalizations of such two-sided $\eps$-optimal estimates to other Sobolev norms  will be published in a subsequent paper~\cite{BBKL20}. 
\end{rem}

\begin{rem}
If $u_0$ has a finite first moment, namely  $\int_{\R^N}u_0(x)|x|\dx<\infty$, then the number $\mu_\Lambda$ defined in \rf{mu-Lambda} is trivially bounded from above by 
$$
\mu_\Lambda=\int_{\R^N}\min\{|x|,\Lambda\}u_0(x)\dx\le \int_{\R^N}u_0(x)|x|\dx\qquad{\rm for \ each\ \ }\Lambda>0.
$$
Thus, assumption \rf{mu-Lambda} is satisfied for each such initial data,  if $\Lambda\kappa_\Lambda\to\infty$ as $\Lambda\to\infty$
and  this property holds true e.g. in the case of 
the kernel $K(x)= -|x|$.
\end{rem}

\subsection*{Notation.}
		We denote by $B_r$ the ball in $\mathbb{R}^N$ centered at $x=0$ with radius $r>0$, and by $\sigma_N={2\pi^{N/2}}/{\Gamma\left(\frac{N}{2}\right)}$  the area of the unit sphere ${\mathbb S}^{N-1}$ in $\R^N$. 
For $p\in [1,\infty]$, the norms of the Lebesgue space $L^p(\mathbb{R}^N)$  are denoted by $\|\cdot\|_p$. 
As usual, we set $H^m=W^{m,2}$, $m>0$, and denote the corresponding homogeneous Sobolev seminorm by $\|\cdot\|_{\dot{H}^m}$. 
We use the analogous notation $\|\cdot\|_{\dot{W}^{m,p}}$ for the homogeneous seminorms in $W^{m,p}$.  
Throughout the paper, the letter  $C$ is used for various positive numbers which may vary from line to line but depend only on the dimension $N$ and the  bounds for the kernel $K$: $\|\nabla K\|_\infty=\|k'\|_\infty$ and $\|k''\|_1$ for $N=1$. 
The dependence upon additional parameters will be indicated explicitly.


\section{Concentration of solutions}\label{sec:cos}

We shall describe a concentration phenomenon at the origin of   solutions  
to problem~\eqref{agg-CP}--\eqref{i-CP}
by considering   the quantity  
\begin{equation}\label{D}
\DD_\Lambda(u(t)) \equiv  
\begin{dcases}
2 u(0,t) &\text{if} \quad N=1,\\
(N-1) \int_{B_{3\Lambda/2}} \frac{u(x,t)}{|x|}\dx &\text{if} \quad N\geq 2,
\end{dcases}
\end{equation}   
where the  scaling parameter  $\Lambda>0$ is chosen in a suitable way, according to the behaviour of the initial condition, see \eqref{mu-Lambda}. 
The following theorem, stating that as $\varepsilon \to 0$, after time averaging, $\DD_\Lambda(u)$ grows at least as $\varepsilon^{-1}$, is one of the main results of our work.

\begin{thm}
 \label{thm:concentration}
Let the assumptions \rf{KN1}--\rf{KN3} hold true, and denote by $u$ a radial, nonnegative, global-in-time solution to problem \eqref{agg-CP}--\eqref{i-CP} with an arbitrary $\varepsilon>0$. 
Assume that  the radial nonnegative initial condition $u_0$ satisfies \rf{PhL10},  as well as    condition \rf{mu-Lambda} for some $\Lambda>0$. 
Then, there exist numbers   $T_\Lambda>0$,  $\cL_\Lambda>0$, and  
$\omega_\Lambda>0$
depending {only} on the dimension $N$,  mass $M$, the number $\mu_\Lambda$ in Assumption~\eqref{mu-Lambda}, and the quantities $\|k'\|_\infty$,   $\kappa_\Lambda$  in assumption \rf{KN2}  (see equations \rf{PhL7} and \rf{N0} below) such that
\begin{equation*}
 \int_0^{T_\Lambda} \DD_\Lambda(u(t)) {\mathrm e}^{-\omega_\Lambda t/\Lambda}\dt \ge \frac{\Lambda \cL_\Lambda}{\varepsilon}.
\end{equation*}
\end{thm}

\begin{proof}
We define the function
\begin{equation}\label{phi}
\varphi(s)=
\begin{cases}
s& \text{if}\quad  0\le s\leq \frac{1}{2},\\
1-\frac12\left(\frac{3}{2}-s\right)^2& \text{if}\quad   \frac{1}{2}\le s \le\frac{3}{2}, \\
1& \text{if}\quad  s\geq  \frac{3}{2},
\end{cases}
\end{equation} 
with  the following properties
\begin{itemize}
\item $0\leq \varphi(s)\leq \min\{s,1\}$ for all $s\geq 0$, 
\item $0\leq \varphi'(s)\leq 1$ for all $s\geq 0$, 
\item  $\varphi''(s)\leq 0$ for all $s\geq 0$ such that $s\neq \frac{1}{2}$ and $s\neq \frac{3}{2}$.
\end{itemize}
For each $\Lambda>0$, we set $\varphi_\Lambda(s) = \varphi(s/\Lambda)$ and we introduce the ``truncated moment'' (compared to the first moment with the  function $|x|$ as weight considered in, e.g., \cite{BKL09}) 
\begin{equation}\label{It}
\cI_\Lambda(t) \equiv \int_{\R^N} \varphi_\Lambda(|x|) u(x,t)\dx \qquad \text{for all}\quad t\ge 0.
\end{equation}
Notice that
\begin{equation}
{\mathcal I}_\Lambda(t)\le M \qquad \text{for all}\quad t\ge 0, \label{I-M}
\end{equation}
 by the mass conservation property \rf{m-cons}   and  properties of $\varphi_\Lambda$.
Our  goal is to derive a differential inequality for $\cI_\Lambda$, see \eqref{dI}  below. 
Thus,  we multiply equation \eqref{agg-CP} by $\varphi_\Lambda(|x|)$, and integrate the resulting identity with respect to $x\in \R^N$. 

Let us show that the contribution of the diffusive term in equation \rf{agg-CP} satisfies the inequality
\begin{equation}
\int_{\R^N} \varphi_\Lambda(|x|)\Delta u(x,t)\dx\le \frac{\DD_\Lambda(u(t))}{\Lambda}  \qquad \text{for all}\quad t\ge 0. \label{PhL2}
\end{equation}
Indeed, if $N\ge 2$, we integrate by parts and use the properties of $\varphi_\Lambda$, positivity of $u$, as well as the fact that $u(x,t)=u(r,t)$ with $r=|x|$ to obtain
\begin{align*} 
\int_{\R^N} \varphi_\Lambda(|x|)\Delta u(x,t)\dx&=-\frac{1}{\Lambda} \int_{\R^N} \varphi'\left( \frac{|x|}{\Lambda} \right) \frac{x}{|x|}\cdot\nabla u(x,t)\dx \\
& =- \frac{\sigma_N}{\Lambda} \int_0^\infty \varphi'\left( \frac{r}{\Lambda} \right) r^{N-1} u_r(r,t) \dr \\
& = \frac{\sigma_N}{\Lambda} \int_0^\infty \left[ \frac{1}{\Lambda} \varphi''\left( \frac{r}{\Lambda} \right)  r^{N-1} + (N-1) \varphi'\left( \frac{r}{\Lambda} \right) r^{N-2} \right] u(r,t)\dr \\
&  \qquad - \frac{\sigma_N}{\Lambda} \left[ \varphi'\left( \frac{r}{\Lambda} \right) r^{N-1} u(r,t) \right]_{r=0}^{r=\infty} \\
&  \le \frac{(N-1)\sigma_N}{\Lambda} \int_0^{3\Lambda/2} u(r,t)r^{N-2}\dr = \frac{\DD_\Lambda(u(t))}{\Lambda}.
\end{align*}
Similarly, when $N=1$, it follows from the symmetry of $u$ that
\begin{align*}
\int_{\mathbb{R}} \varphi_\Lambda(|x|) u_{xx}(x,t) \dx & = 2 \int_0^\infty \varphi_\Lambda(x) u_{xx}(x,t)\dx \\
& = 2 \left[ \varphi\left( \frac{x}{\Lambda} \right) u_x(x,t) \right]_{x=0}^{x=\infty} - \frac{2}{\Lambda} \int_0^\infty \varphi'\left( \frac{x}{\Lambda} \right) u_x(x,t)\dx \\
& = - \frac{2}{\Lambda} \left[ \varphi'\left( \frac{x}{\Lambda} \right) u(x,t) \right]_{x=0}^{x=\infty} + \frac{2}{\Lambda^2} \int_0^\infty \varphi''\left( \frac{x}{\Lambda} \right) u(x,t)\dx \\
& \le \frac{2}{\Lambda} u(0,t) = \frac{\DD_\Lambda(u(t))}{\Lambda} ,
\end{align*}
which completes the proof of inequality \eqref{PhL2}.

Next, we estimate the contribution of the truncated moment of the nonlinear drift term 
\begin{equation*}
J_\Lambda(t) \equiv  - \int_{\R^N} \varphi_\Lambda(|x|)\nabla\cdot \left(u(x,t)\ \nabla K\ast u(x,t)\right)\dx. 
\end{equation*}
Integrating by parts and using the properties of $K$ and $\varphi_\Lambda$, as well as a symmetrization argument, we obtain   
\begin{align}\label{derl1}
J_\Lambda(t) &= \frac{1}{\Lambda} \int_{\R^N} u(x,t) \varphi'\left( \frac{|x|}{\Lambda} \right) \frac{x}{|x|} \cdot (\nabla K\ast u)(x,t)\dx 
\\ \nonumber
&= \frac{1}{\Lambda} \int_{\R^N} \int_{\R^N} u(x,t) u(y,t) \varphi'\left( \frac{|x|}{\Lambda} \right) \frac{x}{|x|} \cdot \nabla K(x-y)\dx \dy  
\\ \nonumber&=\frac{1}{2\Lambda} \int_{\R^N} \int_{\R^N} u(x,t) u(y,t) k'(|x-y|) \Phi_\Lambda(x,y)\dx \dy, 
\end{align}
where
\begin{equation*}
\Phi_\Lambda(x,y) \equiv 
\frac{x-y}{|x-y|} \cdot \left[ \varphi'\left( \frac{|x|}{\Lambda} \right) \frac{x}{|x|} - \varphi'\left( \frac{|y|}{\Lambda} \right) \frac{y}{|y|} \right], \qquad (x,y)\in \mathbb{R}^N\times\mathbb{R}^N.
\end{equation*}
Introducing the quantity 
\begin{equation*}
J_{\Lambda,1}(t) = \frac{1}{2\Lambda} \int_{B_{\Lambda/2}} \int_{B_{\Lambda/2}} u(x,t) u(y,t) k'(|x-y|) \Phi_\Lambda(x,y)\dx \dy,
\end{equation*}
we notice that, for $(x,y)\in B_{\Lambda/2}\times B_{\Lambda/2}$, 
\begin{equation} \label{Phisym}
\Phi_\Lambda(x,y) = \frac{x-y}{|x-y|} \cdot \left( \frac{x}{|x|} - \frac{y}{|y|} \right) = \frac{|x|+|y|}{|x-y|} \left( 1 - \frac{x\cdot y}{|x| |y|} \right) \ge 1 - \frac{x\cdot y}{|x| |y|} \ge 0.
\end{equation}
Moreover, since $|x-y|\le \Lambda$, by assumption \rf{KN2}, we have 
\begin{equation} \label{integ}
k'(|x-y|)\le -\kappa_\Lambda,\quad (x,y)\in B_{\Lambda/2}\times B_{\Lambda/2}.
\end{equation}
Combining \eqref{Phisym} and \eqref{integ} we get 
\begin{align*}
J_{\Lambda,1}(t) & \le - \frac{\kappa_\Lambda}{2\Lambda} \int_{B_{\Lambda/2}} \int_{B_{\Lambda/2}} u(x,t) u(y,t) \Phi_\Lambda(x,y)\dx \dy \\
& \le - \frac{\kappa_\Lambda}{2\Lambda} \int_{B_{\Lambda/2}} \int_{B_{\Lambda/2}} u(x,t) u(y,t) \left( 1 - \frac{x\cdot y}{|x| |y|} \right)\dx \dy.
\end{align*}
Moreover, recalling the identity
\begin{equation*}
\int_{\R^N} x\,\psi(x) \dx=0,
\end{equation*} 
 valid for every radially symmetric function $\psi$, we end up with the inequality 
\begin{equation*}
J_{\Lambda,1}(t) \le - \frac{\kappa_\Lambda}{2\Lambda} \int_{B_{\Lambda/2}} \int_{B_{\Lambda/2}} u(x,t) u(y,t)\dx \dy.
\end{equation*}
Finally, we use the mass conservation property \rf{m-cons}, the inclusion
\begin{equation}
\left( \R^N\times\R^N \right) \setminus \left( B_{\Lambda/2}\times B_{\Lambda/2} \right) \subset \Big( \R^N\times\left(\R^N\setminus B_{\Lambda/2}\right) \Big)\ \cup\ \Big( \left(\R^N\setminus B_{\Lambda/2}\right)\times \R^N \Big), \label{PhL3}
\end{equation}
and the inequality 
\begin{equation}
\frac{1}{2}\leq \varphi_\Lambda(|x|) = \varphi\left( \frac{|x|}{\Lambda} \right) \qquad {\rm for \ all\ \ \ } x\in \mathbb{R}^N \setminus B_{\Lambda/2}, \label{PhL4} 
\end{equation}
to conclude that
\begin{align}
J_{\Lambda,1}(t) & \le - \frac{\kappa_\Lambda}{2\Lambda} M^2 + \frac{\kappa_\Lambda}{2\Lambda} \int _{\left( \R^N\times\R^N \right) \setminus \left( B_{\Lambda/2}\times B_{\Lambda/2} \right)} u(x,t) u(y,t)\dx \dy \nonumber \\
& \le - \frac{\kappa_\Lambda}{2\Lambda} M^2 + \frac{\kappa_\Lambda}{\Lambda} \int_{\mathbb{R}^N} \int _{\R^N \setminus B_{\Lambda/2}} u(x,t) u(y,t)\dx \dy \nonumber \\
& \le - \frac{\kappa_\Lambda}{2\Lambda} M^2 + \frac{2 \kappa_\Lambda}{\Lambda} M \cI_\Lambda(t). \label{PhL5}
\end{align} 
Next, owing to the boundedness of $\varphi'$ and $k'$, 
\begin{equation*}
k'(|x-y|) \Phi_\Lambda(x,y) \le \|k'\|_\infty \left[ \varphi'\left( \frac{|x|}{\Lambda} \right) + \varphi'\left( \frac{|y|}{\Lambda} \right) \right] \le 2 \|k'\|_\infty, \qquad (x,y)\in \mathbb{R}^N\times \mathbb{R}^N,
\end{equation*}
which we combine with relations \eqref{PhL3} and \eqref{PhL4} to estimate $J_\Lambda(t)-J_{\Lambda,1}(t)$, and thereby obtain
\begin{align}
J_\Lambda(t)-J_{\Lambda,1}(t) & = \frac{1}{2\Lambda} \int _{\left( \R^N\times\R^N \right) \setminus \left( B_{\Lambda/2}\times B_{\Lambda/2} \right)} u(x,t) u(y,t) k'(|x-y|) \Phi_\Lambda(x,y)\dx \dy \nonumber \\
& \le \frac{\|k'\|_\infty}{\Lambda} \int _{\left( \R^N\times\R^N \right) \setminus \left( B_{\Lambda/2}\times B_{\Lambda/2} \right)} u(x,t) u(y,t)\dx \dy \nonumber \\
& \le \frac{2 \|k'\|_\infty}{\Lambda} \int_{\mathbb{R}^N} \int _{\R^N \setminus B_{\Lambda/2}} u(x,t) u(y,t)\dx \dy \nonumber \\
& \le \frac{4 \|k'\|_\infty}{\Lambda} M \cI_\Lambda(t). \label{PhL6}
\end{align}

Gathering identity \eqref{derl1} and estimates \eqref{PhL2}, \eqref{PhL5} and \eqref{PhL6}, we get the differential inequality  
\begin{equation} \label{dI}
\Lambda \frac{{\rm d}}{\dt} \cI_\Lambda(t) \leq \varepsilon \DD_\Lambda(u(t)) -\frac{\kappa_\Lambda}{2} M^2 + \omega_\Lambda \cI_\Lambda(t)
\qquad \text{for all}\quad t\ge 0,
\end{equation}
where 
\begin{equation}
\omega_\Lambda \equiv  2 M \left( \kappa_\Lambda + 2 \|k'\|_\infty \right). \label{PhL7}
\end{equation}
Equivalently,
\begin{equation*}
\frac{{\rm d}}{\dt} \left[ \cI_\Lambda(t) {\mathrm e}^{-\omega_\Lambda t/\Lambda} \right] \le \left( \varepsilon \DD_\Lambda(u(t)) -\frac{\kappa_\Lambda}{2} M^2 \right) \frac{{\mathrm e}^{-\omega_\Lambda t/\Lambda}}{\Lambda} \qquad \text{for all}\quad t\ge 0.
\end{equation*}
After an integration with respect to time, we obtain for each $T>0$,
\begin{align*}
- \cI_\Lambda(0) & \le \cI_\Lambda(T) {\mathrm e}^{-\omega_\Lambda T/\Lambda} - \cI_\Lambda(0) \le \frac{1}{\Lambda} \int_0^T \left[ \varepsilon \DD_\Lambda(u(t)) - \frac{\kappa_\Lambda}{2} M^2 \right] {\mathrm e}^{-\omega_\Lambda t/\Lambda}\dt \\
& = \frac{\varepsilon}{\Lambda} \int_0^T \DD_\Lambda(u(t)) {\mathrm e}^{-\omega_\Lambda t/\Lambda}\dt - \frac{\kappa_\Lambda}{2\omega_\Lambda} M^2 \left( 1 - {\mathrm e}^{-\omega_\Lambda T/\Lambda} \right).
\end{align*}
Hence, for each $T>0$ we have
\begin{equation}
\frac{\varepsilon}{\Lambda} \int_0^T \DD_\Lambda(u(t)) {\mathrm e}^{-\omega_\Lambda t/\Lambda}\dt \ge \frac{\kappa_\Lambda}{2\omega_\Lambda} M^2 \left( 1 - {\mathrm e}^{-\omega_\Lambda T/\Lambda} \right) - \cI_\Lambda(0).
 \label{PhL8}
\end{equation}
By  assumption \eqref{mu-Lambda} and the properties of $\varphi$, we get 
\begin{equation*}
\cI_\Lambda(0) \le \frac{1}{\Lambda} \int_{\mathbb{R}^N} \min\{|x|,\Lambda\} u_0(x)\dx = \frac{\mu_\Lambda}{\Lambda} < \frac{\kappa_\Lambda}{2\omega_\Lambda} M^2, 
\end{equation*}
so that choosing the numbers
\begin{equation}\label{N0}
\cL_\Lambda \equiv  \frac{1}{2} \left( \frac{\kappa_\Lambda}{2\omega_\Lambda} M^2 - \cI_\Lambda(0) \right)> 0 
\qquad 
\text{and}\qquad 
T_\Lambda \equiv  \frac{\Lambda}{\omega_\Lambda} \log\left( \frac{\kappa_\Lambda M^2}{2 \omega_\Lambda \cL_\Lambda} \right) > 0,
\end{equation}
we obtain 
\begin{equation}
\frac{\kappa_\Lambda}{2\omega_\Lambda} M^2 \left( 1 - {\mathrm e}^{-\omega_\Lambda T/\Lambda} \right) - \cI_\Lambda(0) \ge \cL_\Lambda \qquad\text{for each}
\qquad 
  T\ge T_\Lambda, \label{PhL9}
\end{equation}
which  completes the proof.
\end{proof}


\section{Upper estimates of  Lebesgue and Sobolev norms} \label{sec:Lp}  

\newcounter{NumConstC}

  In this section, we deal with a global-in-time, nonnegative, regular  solution to problem~\rf{agg-CP}--\rf{i-CP} corresponding to the initial condition $0\le u_0\in L^1(\mathbb{R}^N)\cap L^{\infty}(\mathbb{R}^N)$ (moreover, we require $u_0\in H^1(\R)$ for $N=1$).
For $N \ge 2$, the only assumption on the kernel is $\nabla K\in L^\infty(\mathbb{R}^N)$ and for $N=1$ we use a more specific assumption
 (which can easily be generalized).  
In particular, the kernel $K$ can be smooth and  $\kappa_0=0$. 
Here, we do not require $u$ to be radially symmetric. 
We will use systematically the mass conservation property \rf{m-cons}.

\begin{lem} \label{Lpupper}
Let $u$ be a (not necessarily radial) nonnegative solution to problem~\rf{agg-CP}--\rf{i-CP} with $\eps>0$, with the interaction kernel satisfying $\nabla K\in L^\infty(\mathbb{R}^N)$, and with initial condition $u_0$ such as in \rf{PhL10}. 
For each $p \in [1,\infty)$ there exists a constant $C_p>0$ such that
$$
\|u(t)\|_p \leq \max\left\{ M,\|u_0\|_{\max\{2,p\}},\ C_{p} M^{(N(p-1)+p)/p} \eps^{-N(p-1)/p}\right\}
$$
for  all  $t \geq 0$.
\end{lem}

\begin{proof}
The case $p=1$ is obvious by the conservation of mass \rf{m-cons}. 

For $p\ge 2$, we use the energy method. 
Integrating by parts, and then using the H{\"o}lder and the Young inequalities as well as the assumption $\nabla K\in L^\infty(\mathbb{R}^N)$, we obtain
\begin{align*}
\frac{1}{p(p-1)} \frac{\rm d}{\dt} \|u\|_p^p & = - \eps \int_{\R^N}{|\nabla u |^2 u^{p-2} \dx} + \int_{\R^N} u^{p-1} \nabla u \cdot (\nabla K\ast u) \dx \\
& \leq - \frac{4\eps}{p^2} \|\nabla u^{p/2} \|_2^2  + \frac{2}{p} \|u^{p/2}\|_2 \|\nabla u^{p/2} \|_2 \|\nabla K\ast u\|_\infty \\
& \le \frac{4}{p^2} \|\nabla u^{p/2} \|_2 \left(- \eps \|\nabla u^{p/2} \|_2 + \frac{p}{2} \| \nabla K \|_\infty \|u\|_1 \|u\|_p^{p/2} \right)  \\
& \le \frac{4}{p^2} \|\nabla u^{p/2} \|_2 \left(- \eps \|\nabla u^{p/2} \|_2 + \frac{pM}{2} \| \nabla K \|_\infty \|u\|_p^{p/2} \right).
\end{align*}
It follows from  the H\"older and Gagliardo--Nirenberg inequalities and the mass conservation \rf{m-cons} that 
\begin{align*} 
\|u\|_p^{p/2} = \|u^{p/2}\|_2 & \le C \|\nabla u^{p/2} \|_2^{N/(N+2)} \|u^{p/2}\|_1^{2/(N+2)} \\
& \le C \|\nabla u^{p/2} \|_2^{N/(N+2)} \left( \|u\|_p^{p(p-2)/2(p-1)} M^{p/2(p-1)} \right)^{2/(N+2)} \\
& \le C \|\nabla u^{p/2} \|_2^{N/(N+2)} M^{p/(N+2)(p-1)} \|u\|_p^{p(p-2)/(N+2)(p-1)},
\end{align*} 
and therefore
\begin{equation} \label{GNp}
\|\nabla u^{p/2}\|_2 \ge C M^{-p/N(p-1)} \|u\|_p^{p(N(p-1)+2)/2N(p-1)}.
\end{equation}
Using \eqref{GNp}, we get
\begin{align} \nonumber
\frac{1}{p(p-1)} \frac{\rm d}{\dt} \|u\|_p^p \le\ & \frac{4}{p^2} \|u\|_p^{p/2} \| \nabla u^{p/2 }\|_2
\\ \nonumber
& \times \left(- C \eps M^{-p/N(p-1)} \|u\|_p^{p/N(p-1)} + \frac{pM}{2} \|\nabla K\|_\infty \right)
\\ \nonumber
\le &\ \frac{2M}{p} \|\nabla K\|_\infty \|u\|_p^{p/2} \| \nabla u^{p/2 }\|_2
\\ \label{decr}
& \times \left(- \eps C(p)^{-p/N(p-1)} \|u\|_p^{p/N(p-1)} M^{-(N(p-1)+p)/N(p-1)} + 1\right) 
\end{align}
Let us show that the inequality \eqref{decr} implies the estimate  
\begin{equation}
\|u(t)\|_{p} \leq U_p \equiv \max\left\{ \|u_0\|_p,\ C(p) M^{(N(p-1)+p)/p} \eps^{-N(p-1)/p} \right\}\ \ \ {\rm for\ all\ \ }t>0. \label{PhL202}
\end{equation} 
Indeed, for $\delta>0$, consider the set 
$$
A_{\delta}=\left\{ t \geq 0:\, \|u(t)\|_{p} \leq U_p+\delta \right\}.
$$
Clearly, $0 \in A_{\delta}$ and the time continuity of $u$ in $L^p(\mathbb{R}^N)$ ensures that 
$$
\tau_\delta := \sup\{ t\ge 0:\, [0,t]\subset A_\delta \}\in (0,\infty].
$$
Assume now for contradiction that $\tau_\delta<\infty$. On the one hand, the definition of $\tau_\delta$ implies that 
\begin{equation}
\|u(\tau_\delta)\|_p^p = (U_p + \delta)^p \ge \|u(t)\|_p^p \ \ \ {\rm for\ all\ \ }t\in (0,\tau_\delta). \label{PhL200}
\end{equation}
Hence,
\begin{equation}
\frac{\rm d}{\dt} \|u(\tau_\delta)\|_p^p\ge 0. \label{PhL201}
\end{equation}
On the other hand, we infer from \eqref{decr} and \eqref{PhL200} that 
\begin{align*}
\frac{\rm d}{\dt} \|u(\tau_\delta)\|_p^p & \le \frac{2M}{p} \|\nabla K\|_\infty \|u(\tau_\delta)\|_p^{p/2} \| \nabla u^{p/2}(\tau_\delta) \|_2 \left( - \|u(\tau_\delta)\|_p^{p/N(p-1)} U_p^{-p/N(p-1)} + 1 \right) \\
& = \frac{2M}{p} \|\nabla K\|_\infty \|u(\tau_\delta)\|_p^{p/2} \| \nabla u^{p/2}(\tau_\delta) \|_2 \left( - \left( \frac{U_p+\delta}{U_p} \right)^{p/N(p-1)} + 1 \right) < 0,
\end{align*}
which contradicts \eqref{PhL201}. Consequently, $\tau_\delta=\infty$ and $A_{\delta}=[0,\infty)$ for all $\delta>0$. Letting $\delta\to 0$ completes the proof of \eqref{PhL202}.
\\
The case $1<p<2$ follows then from H{\"o}lder's inequality and the case $p=2$. Indeed, by \eqref{PhL202} with $p=2$,
\begin{align*}
\Vert u(t) \Vert_p & \le M^{(2-p)/p} \Vert u(t) \Vert_2^{(2p-2)/p} \\
& \leq M^{(2-p)/p} \max\left\{ \|u_0\|_2,\ C(2) M^{(N+2)/2} \eps^{-N/2} \right\}^{2(p-1)/p}
\\
& \le \max \left\{ M^{(2-p)/p}\|u_0\|_{2}^{2(p-1)/2},\ C(p) M^{(N(p-1)+p)/p} \eps^{-N(p-1)/p}\right\} \\
& \le \max \left\{ M,\ \|u_0\|_{2},\ C(p) M^{(N(p-1)+p)/p} \eps^{-N(p-1)/p}\right\},
\end{align*}
as claimed.
\end{proof}

In the one-dimensional case, we need an analogous estimate for a Sobolev norm. 
Let us first note a crucial property of the interaction kernel.  

\begin{lem}\label{N=1} 
If $K:\R\to\R$ satisfies assumptions \rf{KN1}--\rf{KN3}, then for each $v\in L^1(\R)\cap L^\infty(\R)$ it follows that 
\begin{equation} \label{jump}
(K'\ast v)_x=2\kappa_0v+k''(|\cdot|)\ast v.
\end{equation}
\end{lem}
We skip the elementary proof of this result which is related to the fact that $K'$ has a jump of size $2\kappa_0$ at the origin.

\begin{lem} \label{H1upper}
\refstepcounter{NumConstC}\label{cstC4} 
Let $u$ be a (not necessarily even) nonnegative solution to problem~\rf{agg-CP}--\rf{i-CP} with $N=1$, with  $\eps>0$ and with initial condition $u_0$ such as in \rf{PhL10}. 
Suppose moreover that $u_0\in H^1(\R)$ and  $ K'\in L^\infty(\mathbb{R})$ has the property~\eqref{jump}.
Then, the following  inequality holds true:
$$
\| u(t) \|_{\dot{H}^1} \leq \max\left\{ \| u_0 \|_{\dot{H}^1},\ C_{\ref{cstC4}} M^{5/2} \eps^{-3/2} \right\} 
\qquad
\text{for all} \quad t\geq 0.
$$
\end{lem}

\begin{proof}
Integrating by parts and using Lemma \ref{N=1} we obtain
\begin{align*}
& \frac{1}{2} \frac{\rm d}{\dt} \Vert u \Vert^2_{\dot{H}^1} = -\eps \Vert u \Vert^2_{\dot{H}^{2}}
-\int_{\R}{u_x(u K' \ast u)_{xx} \dx} \\
&= -\eps \Vert u \Vert^2_{\dot{H}^{2}}-\int_{\R}{u_x u_{xx} (K' \ast u) \dx}-2 \int_{\R}{u_x^2 (K' \ast u)_{x} \dx}
- \int_{\R}{uu_x (K' \ast u_x)_{x} \dx}\\
&= -\eps \Vert u \Vert^2_{\dot{H}^{2}}-\int_{\R}{uu_x (K' \ast u_x)_{x} \dx}-\frac{3}{2} \int_{\R}{u_x^2 (K' \ast u)_{x} \dx}
\\
&=-\eps \Vert u \Vert^2_{\dot{H}^{2}}-5\kappa_0 \underbrace{\int_{\R}{uu_x^2 \dx}}_A - \underbrace{\int_{\R}{uu_x( k''(|\cdot|) \ast u_x) \dx}}_B
-\frac{3}{2} \underbrace{\int_{\R}{u^2_x( k''(|\cdot|) \ast u) \dx}}_E.
\end{align*}
Using the  H{\"o}lder and the Gagliardo--Nirenberg  inequalities, we get
\begin{align*}
|A| &\leq M \Vert u \Vert_{\dot{W}^{1,\infty}}^2 \leq CM \left(M^{1/5} \Vert u \Vert^{4/5}_{\dot{H}^{2}}\right)^2 \leq C M^{7/5} \Vert u \Vert^{8/5}_{\dot{H}^{2}},
\end{align*}
and then applying,  moreover,  the Young inequality,
\begin{align*}
|B| &\leq  \Vert u \Vert_{2} \Vert u \Vert_{\dot{H}^{1}} \Vert  k''(|\cdot|) \Vert_{1} \Vert u_x \Vert_{\infty}\\
& \leq C \left (M^{4/5} \Vert u \Vert^{1/5}_{\dot{H}^{2}}\right)  \left(M^{2/5} \Vert u \Vert^{3/5}_{\dot{H}^{2}}\right) 
\left(M^{1/5} \Vert u \Vert^{4/5}_{\dot{H}^{2}}\right) \\
&\leq C M^{7/5} \Vert u \Vert^{8/5}_{\dot{H}^{2}},
\end{align*}
and
\begin{align*}
|E| &\leq  \Vert u \Vert^2_{\dot{H}^{1}} \Vert k''(|\cdot|) \Vert_{1} \Vert u \Vert_{\infty}  \leq C \left(M^{2/5} \Vert u \Vert^{3/5}_{\dot{H}^{2}}\right) ^2
\left(M^{3/5} \Vert u \Vert^{2/5}_{\dot{H}^{2}}\right) \leq C M^{7/5} \Vert u \Vert^{8/5}_{\dot{H}^{2}}.
\end{align*}
Consequently, we obtain 
\begin{align} \label{preGN}
& \frac{1}{2} \frac{\mathrm{d}}{\mathrm{d}t} \|u\|_{\dot{H}^1}^2 \leq - \eps \|u\|_{\dot{H}^{2}}^2 
+ C M^{7/5} \Vert u \Vert^{8/5}_{\dot{H}^{2}}.
\end{align}
By the  Gagliardo--Nirenberg inequality, we have $\|u\|_{\dot{H}^1} \leq C M^{2/5} \Vert u \Vert^{3/5}_{\dot{H}^{2}}$, which implies that
\begin{align} \label{GNH1}
& \|u\|_{\dot{H}^2} \geq C \|u\|_{\dot{H}^1}^{5/3} M^{-2/3}.
\end{align}

Consequently, inequalities \eqref{preGN}--\eqref{GNH1} yield
\begin{align*} 
\frac{1}{2} \frac{\mathrm{d}}{\mathrm{d}t} \|u\|_{\dot{H}^1}^2 & \leq \Vert u \Vert^{8/5}_{\dot{H}^{2}}
\left(-C \eps \|u\|_{\dot{H}^1}^{2/3} M^{-4/15} + C M^{7/5}\right).
\\
& \leq C \Vert u \Vert^{8/5}_{\dot{H}^{2}} M^{7/5} \left( -C \eps \|u\|_{\dot{H}^1}^{2/3} M^{-5/3}+1\right). 
\end{align*} 
Now we conclude as for the proof of Lemma~\ref{Lpupper}.
\end{proof}


\section{Proofs of main results} \label{sec:proof}

Now we prove the results stated in Section~\ref{sec:mr}

\begin{proof}[Proof of Theorem~\ref{conc}]
\noindent\textit{The case $N\ge 2$.}
We recall that $\Lambda$ is a number satisfying the assumption \rf{mu-Lambda}, which allows us to use Theorem~\ref{thm:concentration}. 
Below, we consider a parameter $\lambda>0$; its value will be specified later.
\\
For arbitrary $T>0$, using twice the H\"older  inequality, 
we get
\begin{equation} \label{PhL20}
\begin{split}
& \int_0^T \int_{B_{\lambda \varepsilon}} \frac{u(x,t)}{|x|} \dx \dt 
\\ 
& \quad \le \left( \int_0^T \int_{B_{\lambda \varepsilon}} \frac{1}{|x|^{(2N-1)/2}} \dx \dt \right)^{2/(2N-1)} \\
& \qquad\qquad \times \left(\int_0^T \int_{B_{\lambda \varepsilon}} u(x,t)^{(2N-1)/(2N-3)} \dx \dt \right)^{(2N-3)/(2N-1)}
\\ 
& \quad \le \left(2 \sigma_N T \sqrt{\lambda \varepsilon} \right)^{2/(2N-1)} \left(\int_0^T \int_{B_{\lambda \varepsilon}} u(x,t)^{(2N+1)/(2N-3)} \dx \dt \right)^{(2N-3)/(4N-2)}
\\ 
& \qquad\qquad \times \left(\int_0^T \int_{B_{\lambda \varepsilon}} u(x,t) \dx \dt \right)^{(2N-3)/(4N-2)}.  
\end{split} 
\end{equation}
Hence, there exists $\eps_\ast>0$ depending on $N$, $M$, and $u_0$ such that for $0<\eps \leq \eps_\ast$, by the $L^p$-estimates in Lemma~\ref{Lpupper},  
\begin{equation}
\begin{split} 
 \int_0^T &\int_{B_{\lambda \varepsilon}} \frac{u(x,t)}{|x|} \dx \dt 
\\ 
& \le \left(\sigma_N T \sqrt{\lambda \varepsilon}\right)^{2/(2N-1)} T^{(2N-3)/(4N-2)} 
\left(\sup_{t \in [0,T]}{\Vert u(t) \Vert_{(2N+1)/(2N-3)}}\right)^{(2N+1)/(4N-2)} 
\\ 
& \qquad\qquad \times \left(\int_0^T \int_{B_{\lambda \varepsilon}} u(x,t) \dx \dt \right)^{(2N-3)/(4N-2)} 
\\ 
& \le C T^{(2N+1)/(4N-2)} \lambda^{1/(2N-1)} \varepsilon^{1/(2N-1)} \varepsilon^{-4N(2N+1)/(4N-2)(2N+1)} 
\\ 
& \qquad\qquad \times \left(\int_0^T \int_{B_{\lambda \varepsilon}} u(x,t) \dx \dt \right)^{(2N-3)/(4N-2)} 
\\ 
& = C T^{(2N+1)/(4N-2)} \lambda^{1/(2N-1)} \varepsilon^{-1}  \left(\int_0^T \int_{B_{\lambda \varepsilon}} u(x,t) \dx \dt \right)^{(2N-3)/(4N-2)} . 
\end{split}
\end{equation} 
On the other hand, by the mass conservation \rf{m-cons},  
\begin{align}
\int_0^T \int_{B_{3\Lambda/2}\setminus B_{\lambda\varepsilon}} \frac{u(x,t)}{|x|} \dx \dt & \le \frac{1}{\lambda\varepsilon} \int_0^T \int_{B_{3\Lambda/2}\setminus B_{\lambda\varepsilon}} u(x,t) \dx \dt \nonumber \\ 
& \le \frac{1}{\lambda\varepsilon} \int_0^T \|u(t)\|_1\dt = \frac{MT}{\lambda\varepsilon}. \label{PhL21}
\end{align}
Recalling the definition of the quantity $\DD_\Lambda(u)$ in \eqref{D}, we deduce from inequalities  \eqref{PhL20} and \eqref{PhL21} that
\begin{equation}  \label{PhL22}
\begin{split}
& \frac{\varepsilon}{N-1} \int_0^T \DD_\Lambda(u(t)) \dt \\
& \quad \le C T^{(2N+1)/(4N-2)} \lambda^{1/(2N-1)} \left(\int_0^T \int_{B_{\lambda \varepsilon}} u(x,t) \dx \dt \right)^{(2N-3)/(4N-2)} + \frac{MT}{\lambda}.
\end{split}
\end{equation}
Next, we  infer from Theorem~\ref{thm:concentration}, the positivity of $\omega_\Lambda$, and inequality~\eqref{PhL22} with $T=T_\Lambda$ that
\begin{equation*}
\frac{\Lambda \mathcal{L}_\Lambda}{N-1} - \frac{MT_\Lambda}{\lambda} \le  C T_\Lambda^{(2N+1)/(4N-2)} \lambda^{1/(2N-1)}
\left(\int_0^{T_\Lambda} \int_{B_{\lambda \varepsilon}} u(x,t) \dx \dt \right)^{(2N-3)/(4N-2)} .
\end{equation*} 
Finally, we choose
\begin{equation}
\lambda \equiv \frac{2(N-1)MT_{\Lambda}}{\Lambda \cL_{\Lambda}} \label{PhL23}
\end{equation} 
to complete the proof of inequality~\eqref{concen} when $N\ge 2$.
\bigskip

\textit{The case $N=1$.} Again, we recall that $\Lambda$ is a number satisfying the assumption \rf{mu-Lambda}, which allows us to use Theorem~\ref{thm:concentration} and we consider a parameter $\lambda>0$ to be specified later. It follows from Theorem~\ref{thm:concentration} and the Cauchy--Schwarz inequality that 
\begin{align*}
\int_0^{T_\Lambda} \int_{-\lambda\varepsilon}^{\lambda\varepsilon} u(x,t) \dx \dt & = \int_0^{T_\Lambda} \int_{-\lambda\varepsilon}^{\lambda\varepsilon} \left( u(0,t) + \int_0^x u_y(y,t) \dy \right) \dx\dt \\
& \ge \int_0^{T_\Lambda} \left( 2 \lambda\varepsilon u(0,t) - 2 (\lambda\varepsilon)^{3/2} \|u_x(t)\|_2 \right) \dt \\
& = \lambda\varepsilon \left( \int_0^{T_\Lambda} \DD_\Lambda(u(t)) \dt - 2 T_\Lambda \sqrt{\lambda\varepsilon} \sup_{t\in [0,T_\Lambda]} \|u(t)\|_{\dot{H}^1}  \right) \\
& \ge \lambda \left( \Lambda \cL_\Lambda - 2 T_\Lambda \sqrt{\lambda} \varepsilon^{3/2} \sup_{t\in [0,T_\Lambda]} \|u(t)\|_{\dot{H}^1} \right).
\end{align*}
We infer from Lemma~\ref{H1upper} that there exists $\eps_\ast>0$ such that for $0<\eps \leq \eps_\ast$, $\|u(t)\|_{\dot H^1}\le C_1M^{5/2}\eps^{-3/2}$, and the above inequality leads to 
\begin{equation*}
\int_0^{T_\Lambda} \int_{-\lambda\varepsilon}^{\lambda\varepsilon} u(x,t) \dx \dt \ge \lambda \left( \Lambda  \cL_{\Lambda} - 2C_1T_\Lambda M^{5/2} \lambda^{1/2} \right).
\end{equation*}
We then complete the proof by choosing $\lambda \equiv ( \Lambda \cL_\Lambda/(4C_1 M^{5/2} T_\Lambda))^2$. 
\end{proof}


Next, we prove the lower estimates for Lebesgue norms.

\begin{proof}[Proof of Corollary~\ref{c2}]
Let $p\in [1,\infty)$. By the H{\"o}lder inequality, Theorem~\ref{conc}, and Lemma~\ref{Lpupper}, 
\begin{align*}
C_* & \le \int_0^{T_*} \int_{B_{\lambda\varepsilon}} u(x,t) \dx \dt \\ 
&  \le \Big( \frac{\sigma_N}{N} \Big)^{\frac{p-1}{p}} (\lambda \varepsilon)^{\frac{N(p-1)}p} \int_0^{T_*} \left( \int_{B_{\lambda\varepsilon}} u(x,t)^p\dx \right)^{1/p} \dt,
\end{align*} 
which proves Corollary~\ref{c2} with $C_{**}(p)=C_* \Big( \frac{\sigma_N}{N} \Big)^{-\frac{p-1}p} \lambda^{-\frac{N(p-1)}p}.
$ 
\end{proof}


\section*{Acknowledgements}
We thank  Mi{\l}osz Krupski for interesting comments. 
This work was partially supported by the French-Polish PHC Polonium grant 40592NJ, and NCN 2016/23/B/ST1/00434 (the first named author). 


\bibliographystyle{siam}
\bibliography{BBKL-biblio}

\end{document}